\documentclass[11pt]{article}
\usepackage{amsmath,amsthm,amsfonts,latexsym,amssymb,enumerate,color}
\usepackage{graphicx}

\def\CC{\mathbb C}

\def\MM{{\mathcal M}}
\def\RR{\mathbb R}

\def\LH{{\rm L}(H)}
\def\LUX{{\rm L}(U,X)}
\def\LX{{\rm L}(X)}
\def\LL{\mathcal L}

\def\imag{\mathop{\rm Im}\nolimits}
\def\re{\mathop{\rm Re}\nolimits}

\def\sgn{\mathop{\rm sgn}\nolimits}
\def\ds{\displaystyle}

\newtheorem{thm}{Theorem}[section]
\newtheorem{prop}[thm]{Proposition}

\newtheorem{cor}[thm]{Corollary}
\newtheorem{rem}[thm]{Remark}

\newtheorem{ex}[thm]{Example}

\numberwithin{equation}{section}

\def\beginpf{\begin{proof}}
\def\endpf{\end{proof}}
\def\beq{\begin{equation}}
\def\eeq{\end{equation}}

\begin{document}

\title{Weighted operator-valued function spaces applied to the   stability of delay systems}
\author{Asmahan E. Alajyan\thanks{Department of Mathematics and Statistics, College of Science, King Faisal University, Al-Ahsa, Saudi Arabia and School of Mathematics, University of Leeds, Leeds LS2 9JT, U.K.
{\tt ml15a3a@leeds.ac.uk}} \ and Jonathan R. Partington\thanks{
School of Mathematics, University of Leeds, Leeds LS2 9JT, U.K. {\tt j.r.partington@leeds.ac.uk}}}

\date{July 2020, revised June 2021}

\maketitle

\begin{abstract}
\noindent This paper extends the theory of Zen spaces (weighted Hardy/Berg\-man spaces on the right-hand
half-plane) to the Hilbert-space valued case, and describes the multipliers on them; it is shown
that the methods of $H^\infty$ control can therefore be extended to a family of weighted $L^2$
input and output spaces. Next, the particular case of retarded delay systems with operator-valued transfer
functions is analysed, and the dependence of $H^\infty$ structure on the delay is determined by
developing an extension of the Walton--Marshall technique  used in the scalar case. The method
is illustrated with examples.
\end{abstract}


\noindent {\bf 2010 Subject Classification:} 30H10, 30H20, 44A10, 46E15, 47N70, 93B36

\section{Introduction}

There is an extensive literature on the use of the $H^\infty$ norm of an analytic (operator-valued) function on the right-hand
half-plane $\CC_+$, which describes the gain of a linear time-invariant system from (vector-valued) $L^2(0,\infty)$ inputs to
$L^2(0,\infty)$ outputs; we mention here some well-known books on the subject, namely,
\cite{francis,glover,vidyasagar}.  Some new contributions to the theory will be presented here.
There are two themes: first we extend the recent theory of Zen spaces \cite{zen09,zen10,jpp13} to  functions
taking values in a Hilbert space $H$,
where the Laplace transform $ \LL$ provides an isometric embedding from a weighted
function space $L^2(0,\infty, w(t) dt; H)$ into a space $A^2_\nu(\CC_+,H)$ of analytic operator-valued functions (all notation will be defined below).
In this context we prove a theorem showing that the $H^\infty$ norm  can be used to measure the gain (operator norm)
of a linear system defined
 in the
context of a wide variety of weighted $L^2$ spaces; thus we show that various
notions of stability are equivalent.\\

  Second, we apply the analysis to one particular case, namely
  the 
$H^\infty$ stability of retarded delay systems with transfer
functions of the form
\beq\label{eq:gs}
G(s)=(P(s)I+Q(s)e^{-sh}A)^{-1},
\eeq
where $P$ and $Q$ are real polynomials and $A$ is a bounded operator on a Banach space $X$.
These arise from delay-differential equations of the form
\[
\sum_{j=0}^n a_j \frac{\partial^j x(t)}{\partial t^j} + A \sum_{k=0}^n b_k \frac{\partial^k x(t-h)}{\partial t^k}=u(t),
\qquad x(0)=0,
\]
where the $a_j$ and $b_k$ are scalars,
by taking Laplace transforms of each side, 
so that $\LL x(s)=G(s)\LL u(s)$ where $G$ has the form given in \eqref{eq:gs}. We refer to \cite{BC,partington2004} for further details and other ways of formulating such systems by differential equations.

The example that will illustrate most of our results arises from the delay-differential equation 
\begin{equation}\label{eq:diffeq}
\dot{x}(t)+Ax(t-h)=u(t), \quad x(0)=0,
\end{equation}
and similar equations, with $x(t)\in X, u(t)\in U$ (where $X$ ~and~ $ U $ are Banach spaces) and $h$ is the delay.  
In this case the operator-valued transfer function is $(sI+e^{-sh}A)^{-1}$ and $L^2$-to-$L^2$ stability holds
if and only if the operator-valued function is bounded in the right-hand half-plane $\CC_+$.
The Walton--Marshall method \cite{WM,partington2004} gives such an analysis in the  purely scalar case $A=a$, say: the method 
involves increasing $h$ to see where the zeros of $P(s)+aQ(s)e^{-sh}$
cross the imaginary axis.
Additionally, at such points  the direction in which the zeros  cross the axis can be  identified.
Here we develop the Walton--Marshall method further to apply it  to  bounded operators.\\

We recall that
a classification of delay systems into retarded, neutral and advanced type can be found in \cite{BC}.
We shall show that, even in the operatorial case, for systems of retarded type ($\deg P > \deg Q$) invertibility of $P(s)I+Q(s)e^{-sh}A$ is equivalent to the inverse
function being in $H^\infty$: this is true for retarded systems, but not for systems of neutral type ($\deg P = \deg Q$). Systems of advanced type ($\deg P<\deg Q$) are never stable. \\

We write $H^\infty$ for the Hardy space of bounded analytic functions on the right-hand half-plane $\CC_+$, 
$\LUX$ for the bounded operators from $U$ to $X$,  
and $H^\infty(\CC_+,\LUX)$ or simply
$H^\infty(\LUX)$ for the space of bounded $\LUX$-valued functions, with norm
\[
\|F\|= \sup_{s \in \CC_+} \|F(s)\|.
\]
We shall mostly be able to take $U=X$, and then we write $\LX$ for ${\rm L}(X,X)$.

\section{Stability on weighted $L^2$ spaces}\label{sec:2}

In this section we show that $H^\infty$ methods can be
applied to stability questions in a wide variety of weighted $L^2(0,\infty)$ spaces. 

Let $w(t)$
be a positive measurable function. Then for a separable Hilbert space $H$ we write $L^2(0,\infty,   w(t) dt, H)$
for the space of measurable $H$-valued functions $f$ such that the norm $\|f\|$, given by
\[
\|f\|^2 = \int_0^\infty \|f(t)\|^2 w(t)\, dt,
\]
is finite. We start by giving conditions on $w$ for  the Laplace transform to  induce an isometry
between $L^2(0,\infty,   w(t) dt, H)$ and a space of $H$-valued analytic functions on $\CC_+$.

Let $\nu$ be a positive regular Borel measure satisfying the doubling condition
\[
R:= \sup_{t>0} \frac{\nu[0,2t)}{\nu[0,t)} < \infty.
\]
The Zen space $A^2_\nu(H)$ is defined to consist of all analytic $H$-valued functions $F$ on $\CC_+$
such that the norm, given by
\[
\|F\|^2=\sup_{\epsilon>0} \int_{\overline{\CC_+}} \|F(s+\epsilon)\|^2 \, d\nu(x) \, dy
\]
is finite, where we write $s=x+iy$ for $x\ge 0$ and $y \in \RR$.

The best-known examples here are:
\begin{enumerate}
\item For  $\nu=\delta_0$, a Dirac mass at $0$,  we obtain the
Hardy space $H^2(\CC_+,H)$;
\item For $\nu$ equal to Lebesgue measure ($dx$), we obtain the Bergman space $A^2(\CC_+,H)$.
\end{enumerate}

Often we shall have $\nu\{0\}=0$, in which case $\|F\|^2$ can be written simply as
\[
\int_{\overline{\CC_+}} \|F(s)\|^2 \, d\nu(x) \, dy.
\]

\begin{thm}
Suppose that $w$ is given as a weighted Laplace transform
\beq \label{eq:defw}
w(t)=2\pi \int_0^\infty e^{-2rt} \, d\nu(r), \qquad (t >0).
\eeq
Then the Laplace transform provides an isometric map 
\beq
\LL:  L^2(0,\infty,   w(t) dt, H) \to A^2_\nu(H). \label{eq:ltla}
\eeq
\end{thm}
\beginpf
This result was given in the scalar case $H=\CC$ in \cite{jpp13} (see also \cite{jpp14}, where applications to admissibility and controllability were given, and \cite{zen09,zen10} for earlier related work). The general case follows using the standard method for proving the Hilbert space-valued
case of Plancherel's theorem \cite[Thm. 1.8.2]{ABHN}: let $(e_n)_{n=1}^\infty$ be an orthonormal basis for $H$, and
write 
\[
f(t) = \sum_{n=1}^\infty f_n(t) e_n,
\]
where $f_n \in L^2(0,\infty, w(t) dt, \CC)$. Then $F:=\LL f= \sum_{n=1}^\infty F_n e_n$, where
$F_n = \LL f_n \in A^2_\nu(\CC)$ and $\|f_n\|=\|F_n\|$ from \cite[Prop. 2.3]{jpp13}. 

Now $\|f\|^2 = \sum_{n=1}^\infty \|f_n\|^2$ and $\|F\|^2 = \sum_{n=1}^\infty \|F_n\|^2$,
so the result follows.
\endpf

In the case that $\nu=\delta_0$, 
we have the vectorial version of the well-known Paley--Wiener result linking $L^2(0,\infty)$ and
the Hardy space $H^2(\CC_+)$; for $\nu$ equal to Lebesgue measure, we recover the fact that the weighted signal space
$L^2(0,\infty, dt/t)$ is isometric (within a constant) to the Bergman space on $\CC_+$.\\

We now have a result for input--output stability.

\begin{thm}
Let $G \in H^\infty(\CC_+,\LH)$. Then the multiplication operator $M_G$ defined by
\[
(M_G F)(s) = G(s) F(s) \qquad (s \in \CC_+, \quad F \in A^2_\nu(H))
\]
is bounded on $A^2_\nu(H)$ with $\|M_G\| \le \|G\|_\infty$. In the case when
the Laplace transform \eqref{eq:ltla}
is surjective onto $A^2_\nu(H)$ we have equality.
\end{thm}
\beginpf
It is clear that
\[
\sup_{\epsilon>0} \int_{\overline{\CC_+}}\|G(s+\epsilon)\|^2 \|F(s+\epsilon)\|^2 \, d\nu(x) \, dy
\le \|G\|_\infty^2 \sup_{\epsilon>0} \int_{\overline{\CC_+}}  \|F(s+\epsilon)\|^2 \, d\nu(x) \, dy,
\]
so that $\|M_G\| \le \|G\|_\infty$.

For the converse inequality 
we begin by noting that by \eqref{eq:defw} we have the inequality $w(t) \ge 2\pi e^{-2\epsilon t} \nu[0,\epsilon)$
for every $\epsilon > 0$. 
Hence, if $z=x+iy \in \CC_+$, we have for $0 < \epsilon < x$ the inequality
\[
\int_0^\infty |e^{-\overline zt}/w(t)|^2 \,w(t) dt  \le \int_0^\infty e^{-2xt} \frac{1}{2\pi\nu[0,\epsilon)} e^{2\epsilon t} \, dt < \infty.
\]
Thus the function $k_z:t \mapsto e^{-\overline zt}/w(t)$ lies in $L^2(0,\infty, w(t) dt)$
for every $z \in \CC_+$, and 
we have
\[
\LL f(z) = \langle f, k_z \rangle_{L^2(0,\infty, w(t) dt)}
\]
for all $f \in L^2(0,\infty, w(t) dt)$. That is, $A^2_\nu=\LL L^2(0,\infty, w(t) dt)$ is 
a reproducing kernel Hilbert space with kernel $K_z:=\LL k_z$ (see, for example, \cite{paulsen} for more
on such spaces).
For $x \in H$ we write $K_z \otimes x$ for the function $s \mapsto K_z(s)x$ in
$A^2_\nu(H)$ and note that for a function $F \in A^2_\nu(H)$ we have 
$\langle F, K_z \otimes x\rangle_{A^2_\nu(H)} = \langle F(z), x \rangle_H$.
Moreover $\|K_z \otimes x\|_{A^2_\nu(H)}=\|K_z\|_{A^2_\nu} \|x\|_H$.

Now for $F  \in A^2_\nu(H)$  and $G \in H^\infty(\LH)$ we have,
for every $x \in H$ and $z \in \CC_+$, that 
\begin{eqnarray*}
\langle F,M_G^* ( K_z\otimes x) \rangle_{A^2_\nu(H)} &=& \langle M_GF,  K_z\otimes x \rangle_{A^2_\nu(H)} = 
\langle G(z)F(z),x \rangle_H \\
& = & 
\langle F(z), G(z)^* x \rangle_H =
\langle F, K_z\otimes G(z)^*x \rangle,
\end{eqnarray*}
and so $M_G^* (K_z \otimes x)=  K_z\otimes G(z)^* x$, and $\|M_G\|=\|M_G^*\| \ge \|G^*\|_\infty= \|G\|_\infty$.
\endpf

This result will be fundamental to the analysis in the next section.

\section{Stability of retarded delay systems}\label{sec:3}

\subsection{General results}

In Section \ref{sec:2} we gave necessary and sufficient conditions for an operator-valued function to give a bounded operator on various function spaces.
As indicated in the introduction, we shall apply these results in  the context of delay equations.
We begin by reviewing the classical (scalar) case, before providing a generalization to the
operator-valued case.

The 1980s result of Walton and Marshall on the location 
of the zeros of a scalar function $G(s)=P(s)+Q(s)e^{-sh}$ is the following.

\begin{prop}\label{prop:11}
\cite{WM},\cite[p.132]{partington2004}
	Let $P(s)$ and $Q(s)$ be real polynomials. If
	\begin{equation}
	 P(s)+Q(s){\rm e} ^{-sh},
	\label{formula}\end{equation}  where $h>0$, has a zero at point $s\in i\mathbb{R}$,  then such an 
	$s$ satisfies the equation 
	\begin{equation} P(s)P(-s) = Q(s)Q(-s). \label{prop1}\end{equation}	
	Moreover, if  $P(s), Q(s)$ are not zero at $s$, then the direction in which the zeros cross the axis with increasing $h$ is given by 
	$$\sgn \re\frac{ds}{dh}=\sgn \re \frac{1}{s} \left[\frac{Q'(s)}{Q(s)}-\frac{P'(s)}{P(s)}\right].$$
\end{prop}

Our first result analyses the $H^\infty$ stability of an operator-valued transfer function
$(P(s)+Q(s)e^{-sh}A)^{-1}$, linking it to
to properties of the function $P(s)+\lambda Q(s)e^{-sh}$, where $\lambda$ is
in the spectrum $\sigma(A)$ of $A$. 

\begin{thm}\label{theorem H infty}
	If $A$ is a bounded operator on a Banach space $X$,  and $h\geq 0$ and   $P(s), Q(s)$   complex polynomials with  ~  $\deg P > \deg Q$,  then the following three statements are equivalent:
	\begin{enumerate}[(i)]
		\item $\left(P(s)I+Q(s)A{\rm e}^{-sh} \right)^{-1}\in H^\infty(\LX).$
		\item $\left(P(s)I+Q(s)A{\rm e}^{-sh} \right)$ is invertible $\forall s \in \overline{\mathbb{C}}_+.$  
		\item $P(s)I+\lambda Q(s) {\rm e}^{-sh} \neq 0 ~~\forall s \in \overline{\mathbb{C}}_+, ~\forall \lambda \in \sigma(A).$ 
	\end{enumerate}
\end{thm}

\begin{proof}
	$(i)\implies (ii):$ This is clear.
	
	$(ii)\implies (iii):$	the operator $\left(P(s)I+Q(s)A{\rm e}^{-sh} \right)$ is invertible 
	  if and only if $ 0 \notin \sigma\left[P(s)I+Q(s)A{\rm e}^{-sh} \right] $; but for fixed $s$, we get
	  \[
	  \sigma\left[P(s)I+Q(s)A{\rm e}^{-sh} \right]=\left\{P(s)I+Q(s)\lambda{\rm e}^{-sh}: ~\lambda\in \sigma(A)\right\},
	  \]
	   which means that $P(s)I+Q(s)\lambda{\rm e}^{-sh} \neq 0 ~~\forall s \in \mathbb{C}_+, ~\forall \lambda \in \sigma(A).$ 
	
	$(iii)\implies (i):$ Suppose $P(s)I+Q(s)\lambda{\rm e}^{-sh} \neq 0 ~~\forall s \in \overline{\mathbb{C}}_+, ~\forall \lambda \in \sigma(A)$ and so  $\left(P(s)I+Q(s)A{\rm e}^{-sh} \right)$ is invertible $\forall s \in \overline{\mathbb{C}}_+.$ We show that the inverse is bounded as a function of $s$.
	
	First: there is an  $R>0$ such that for $s \in \overline{\mathbb{C}}_+$ with  $|s|>R$, we have 
	$$|P(s)|>|Q(s)|\|A\||{\rm e}^{-sh}|+1,$$
	and so for $x\in X$ we get
	$$\|P(s)x\|>\left(|Q(s)|\|A\||{\rm e}^{-sh}|\right)\|x\|+\|x\|,$$
	and so 
	\begin{align*}
	\|P(s)Ix+Q(s)A{\rm e}^{-sh}x\|&\geq	\|P(s)x\|-\left(|Q(s)|\|A\||{\rm e}^{-sh}|\right)\|x\|\geq\|x\|.
	\end{align*}
	That means
	$$\|\left(P(s)I+Q(s)A{\rm e}^{-sh}\right)^{-1}\| \leq 1,$$
	and so ~  $\left(P(s)I+Q(s)A{\rm e}^{-sh} \right)^{-1}$ is bounded for $|s|> R,~ s \in \overline{\mathbb{C}}_+$.
	
	Second, to prove $\left(P(s)I+Q(s)A{\rm e}^{-sh} \right)^{-1}$ is uniformly bounded for $|s|\leq  R,~ s \in \overline{{\mathbb{C}}}_+$, we suppose not, so  $\exists (x_n) \subset X,~ \|x_n\|=1$ and a sequence $(s_n) \subset S$ where $S=\{s\in \mathbb{\overline{C}}_+: |s|\leq R\}$ such that 
	$$\left(P(s_n)I+Q(s_n)A{\rm e}^{-s_nh} \right)x_n \rightarrow 0,$$  
	and because $S$ is a compact set then there is a subsequence  $(s_{n_k})_{k \geq 0}$ and $s_0\in S$ such that $(s_{n_k})\rightarrow s_0$. Now 
	\[
	\|P(s_n)I+Q(s_n)A{\rm e}^{-s_nh} -P(s_0)I-Q(s_0)A{\rm e}^{-s_0h} \|\rightarrow 0
	\]
	 and so $\left(P(s_0)I+Q(s_0)A{\rm e}^{-s_0h} \right)x_n \rightarrow 0$, which means that 
	$$0\in \sigma\left(P(s_0)I+Q(s_0)A{\rm e}^{-s_0h} \right),$$
	so there exists a $ \lambda\in \sigma(A)$ such that $P(s_0)I+Q(s_0)\lambda{\rm e}^{-s_0h} =0.$
	
\end{proof}	

\begin{rem}{\rm
(i) The result does not hold in general if $A$ is unbounded (and in this case the linear system may even be
destabilised by an arbitrarily small delay). For example, if $A$ is a diagonal operator on a Hilbert space  
with orthonormal eigenvectors and eigenvalues $\lambda_n=ni + 1/n$ $(n \ge 1)$, then
$s+ \lambda_n$ has no zeros in $\CC_+$ but $(sI+A)^{-1}$ is unbounded on $\CC_+$.\\

(ii)
The location of the poles of a neutral delay
system ($\deg P=\deg Q$) does not determine its stability; as the following
example  indicates \cite{BonnetPartington2002}.

	Consider $\ds G(s) = \frac{1}{s + 1 + s{\rm e}^{-s}}$. If
	$\re s>0$ then we cannot have ${\rm e}^{-s} = -1 -\displaystyle\frac{1}{s}$, since the
	left-hand side has modulus ${<1}$ and the right-hand side
	has modulus strictly ${>1}$; thus this system has no poles in
	$\mathbb{C}_+$, nor indeed on $i\mathbb{R}$ (as is easily verified), although it
	does have a sequence of poles $z_n$ with $\imag z_n  \approx  (2n + 1)\pi$
	and $\re z_n \rightarrow 0$. The system is not   stable, as an
	analysis of its values at $s = i\left[(2n + 1)\pi + \displaystyle\frac{1}{(2n + 1)\pi}\right],
	n\in \mathbb{Z} $, shows that it is not   in $H^\infty$.
	}
\end{rem}

Since  real matrices may have complex spectrum, we require a complex version of Proposition
\ref{prop:11}, as follows:

\begin{prop}\label{prop:2.3}
		Let $P(s)$ and $Q(s)$ be real polynomials. If
	$$P(s)+\lambda  Q(s){\rm e} ^{-sh}$$
    has a zero for some $h \in \RR$, $\lambda\in \mathbb{C}$ and  $s\in i\mathbb{R}$,  then such an $s$ satisfies the equation 
\beq\label{eq:pplqq}
P(s)P(-s) =|\lambda|^2 Q(s)Q(-s).
\eeq	
	Moreover if $P(s), Q(s)$ are not zero at $s$ and $\lambda \ne 0$, then we have 
\beq
\label{eq:cross}
\sgn \re\frac{ds}{dh}=\sgn \re \frac{1}{s} \left[\frac{Q'(s)}{Q(s)}- \frac{P'(s)}{P(s)}\right].
\eeq
\end{prop}
\beginpf
From $P(s)+\lambda Q(s)e^{-sh}=0$, we obtain by taking complex conjugates and noting that $\overline s=-s$, that
$P(-s)+\overline\lambda Q(-s)e^{sh}=0$. This establishes \eqref{eq:pplqq} on eliminating the exponential term from the equations.

Next,
elementary calculus (performing differentiation with respect to $h$) gives us 
\[
(P'(s)+\lambda Q'(s) e^{-sh}-h\lambda Q(s)e^{-sh}) \frac{ds}{dh} - \lambda sQ(s) e^{-sh} = 0,
\]
and using the fact that $\lambda e^{-sh}=-P(s)/  Q(s)$ gives us
\[
\left[ P'(s)-\frac{P(s)Q'(s)}{Q(s)}+P(s)h  \right]  \frac{ds}{dh} = -sP(s) 
\]
or
\[
\frac{ds}{dh}    = -s\left[\frac{ P'(s)}{P(s)}-\frac{ Q'(s)}{Q(s)}+ h \right] ^{-1}
\]
Now $\sgn \re u= \sgn \re u^{-1}$ for every $u \ne 0$, and $h/s$ is purely imaginary, so the result follows.
\endpf

\begin{rem}{\rm
Although it is physically less relevant, we may also consider the case of complex polynomials $P$ and $Q$.
In this case \eqref{eq:pplqq} is replaced by $|P(s)|^2=|\lambda|^2|Q(s)|^2$.
}
\end{rem}

\subsection{Matrices, normal and subnormal operators}

In many of the systems occurring in applications we can determine the spectrum of $A$, and
this makes the analysis above easier to perform explicitly. For systems with inputs and
outputs in $\CC^n$ or $\RR^n$, rather than general Hilbert spaces, we have that $A$ is
a matrix, while, as we explain below, there are  classes of operators on
infinite-dimensional spaces for which the analysis can also be performed directly.\\

A classical theorem of Schur (see, for example, \cite[p. 79]{HornJohnson1990}) asserts that
any finite square matrix can be transformed into an upper triangular matrix by conjugation with a unitary matrix.
This means that when working with
norm estimates for $G(s)$ as in \eqref{eq:gs} we may perform calculations using upper triangular matrices
(and the spectrum is the set of diagonal elements when the matrix is in triangular form).\\

Another important class of infinite-dimensional operators $A$ is the class of normal operators,
such that $A^*A=AA^*$. These are unitarily equivalent to multiplication operators
$M_g: f \mapsto fg$ on an $L^2(\Omega, \mu)$ space, and the spectrum is simply the closure
of the range of $M_g$. We shall illustrate this by an example in Section \ref{sec:ex}.\\

Going beyond that we may consider the class of subnormal operators $A$, for which good references are
\cite{bram,conway}. These may be regarded up to unitary equivalence as restrictions of normal operators $N$ to 
invariant subspaces $\MM$ (that is, $N(\MM) \subseteq \MM)$; as for example the unilateral shift operator,
or multiplication by the independent variable on the Hardy space $H^2$ of the disc.
Such operators have been considered in a systems-theory context in \cite{JMPW}.

A subnormal operator $A$ has a minimal normal extension $N$ (that is, no proper restriction of $N$ is a normal
extension of $A$). We mention this because we then have $\sigma(N) \subseteq \sigma(A) \subseteq \sigma(N) \cup H(N)$,
where $H(N)$ is the union of the bounded components of $\CC \setminus \sigma(N)$ (that is, the ``holes'' in $\sigma(N)$).

We therefore have the following immediate Corollary of Theorem \ref{theorem H infty}.
\begin{cor}
	Under the hypotheses of Theorem \ref{theorem H infty} if $N$ is the minimal normal extension of $A$ and $P(s)+Q(s)\lambda{\rm e}^{-sh}\neq 0$  for all $\lambda \in \sigma(N)\cup H(N)$ then  $(P(s)+Q(s)A{\rm e}^{-sh})^{-1} \in H^\infty(\LX)$. Conversely, if  $(P(s)+Q(s)A{\rm e}^{-sh})^{-1} \in H^\infty(\LX)$ then $P(s)+Q(s)\lambda{\rm e}^{-sh}\neq 0$  for all $\lambda \in \sigma(N)$.
\end{cor}

\subsection{Examples}
\label{sec:ex}

We conclude by showing some examples in which the results of Sections \ref{sec:2}
and \ref{sec:3} can be applied directly to stability analysis.

\begin{ex}{\rm 
	Let $A=\begin{bmatrix}
	1       & 0&0 \\
	0&2      & 1  \\
	0&0&2\\ \end{bmatrix}$. To apply
	Theorem \ref{theorem H infty},
	 with $P(s)=s, Q(s)=1$ and the eigenvalues $\lambda_k\in \{1,2\} $ we   check the zero sets of $s+e^{-sh}$ and $s+2e^{-sh}$.
The equations \eqref{eq:pplqq} giving the points where zeros cross the axis with increasing $h$ 
are $-s^2=1$ and $-s^2=4$, respectively, and 
from $s+ \lambda e^{-sh}=0$
we arrive at stability ranges $[0,\pi/2)$ and $[0,\pi/4)$, respectively.
Thus the system \eqref{eq:diffeq} is stable for $0 \le h < \pi/4$ (it is easily verified using
\eqref{eq:cross} that poles move from left to right as $h$ increases).
}
\end{ex}

\begin{ex}{\rm 
	For the normal matrix	$A=\begin{bmatrix}
	1       & -1 \\
	1      & 1  \\\end{bmatrix}$ we have the transformation  
	$$T=R^{-1}\begin{bmatrix}
	1       & -1 \\
	1      & 1  \\\end{bmatrix} R=\begin{bmatrix}
	1 +i      & 0 \\
	0      & 1-i  \\\end{bmatrix},$$	
	
where 	$R=\frac{1}{\sqrt 2}\begin{bmatrix}
1       & -i \\
-i      & 1  \\\end{bmatrix}$, which is unitary.

Now we have to consider the equation \eqref{eq:pplqq}, obtaining $-s^2=2$ and for each eigenvalue $\lambda$ we
perform the analysis for $s+\lambda e^{-sh}=0$. Because the $\lambda$ are not real, we obtain different
values of $h$ for each eigenvalue, which we summarise now:
\begin{itemize}
\item For $\lambda=1+i$, $s=i\sqrt{2}$, we have $h=\frac{3\pi}{4\sqrt 2}$.
\item  For $\lambda=1+i$, $s=-i\sqrt{2}$, we have $h=\frac{\pi}{4\sqrt 2}$.
\item For $\lambda=1-i$, $s=i\sqrt{2}$, we have $h=\frac{\pi}{4\sqrt 2}$.
\item  For $\lambda=1-i$, $s=-i\sqrt{2}$, we have $h=\frac{3\pi}{4\sqrt 2}$.
\end{itemize}
Again it may be checked that the zeros cross from left to right with increasing $h$.	
	Thus, we can deduce that system (\ref{eq:diffeq}) is stable when $0\leq h< \frac{\pi}{4 \sqrt 2}$.
	}
\end{ex} 

\begin{ex}{\rm
Examples involving subnormal operators are  hard\-er to analyse, since
the spectrum of a non-normal subnormal operator cannot be contained in any simple closed
curve \cite{putnam}.
As an example, consider again 
$P(s)I+Q(s)Ae^{-sh}$, and let  $P(s)=s+1$, $Q(s)=1$, and $A=1+S$, where
$S$ is the unilateral shift  operator; thus $A$ is unitarily equivalent to the operator of multiplication
by $1+z$ on the Hardy space $H^2$ and has spectrum $\sigma(A)=\{z \in \CC: |z-1| \le 1\}$.

From Proposition \ref{prop:2.3} we have at a zero-crossing for $\lambda$, that $1-s^2 = |\lambda|^2$, and thus
we need only consider the points $\lambda\in \sigma(A)$ with $1 \le |\lambda| \le 2$ (a crescent-shaped set).
We see from Proposition \ref{prop:2.3} that the system is stable for $h=0$ and that the poles cross from left to right as we increase $h$.

Setting $s=iy$, with $-\sqrt{3} \le y \le \sqrt{3}$, we then need to work with the 
equation
$e^{-sh}=-(s+1)/\lambda$, or equivalent $e^{sh}=(s-1)/\overline\lambda$.

For each values of $s$ is easily verified that the $\lambda$ leading to the smallest stability margin lies on the circle $\{z \in \CC: |z-1| \le 1\}$, and then it is an exercise to verify that the extreme value $\lambda=2$ corresponding to $s=i\sqrt{3}$
gives the minimal margin of $0 \le h < 2\pi/(3\sqrt 3)$.
 }
 
\end{ex}

%
%

\end{document}